\documentclass[reqno, 11pt]{article}
\pdfoutput=1

 \usepackage{amsthm, amssymb, amsmath}
 \usepackage{fullpage}
\usepackage{verbatim}
\usepackage{graphicx}
\usepackage{epsfig}
\usepackage{color}
\usepackage[all]{xy}

\newtheorem{thm}{Theorem}[section]
\newtheorem*{thm*}{Theorem}
\newtheorem{lem}[thm]{Lemma}
\newtheorem{cor}[thm]{Corollary}
\newtheorem{prop}[thm]{Proposition}

\newtheorem*{conjecture*}{Conjecture}

\theoremstyle{remark} 
\newtheorem*{question*}{Question}

\newtheorem{remark}[thm]{Remark}
\newtheorem{example}[thm]{Example}

\theoremstyle{definition}

\numberwithin{equation}{section}  

\newcommand{\OO}{\mathcal{O}}    
\newcommand{\RR}{\mathbb{R}}     
\newcommand{\PP}{\mathbb{P}}      
\newcommand{\CC}{\mathbb{C}}      

\newcommand{\be}{\begin{equation}}
\newcommand{\ee}{\end{equation}}
\newcommand{\benn}{\begin{equation*}}
\newcommand{\eenn}{\end{equation*}}
\newcommand{\ba}{\begin{aligned}}
\newcommand{\ea}{\end{aligned}}
\newcommand{\bbm}{\begin{bmatrix}}
\newcommand{\ebm}{\end{bmatrix}}
\newcommand{\bpm}{\begin{pmatrix}}
\newcommand{\epm}{\end{pmatrix}}
\newcommand{\bi}{\begin{itemize}}
\newcommand{\ei}{\end{itemize}}
 
\newcommand{\ord}{\operatorname{ord}}
\newcommand{\supp}{\operatorname{supp}}   

\newcommand{\Berk}{\mathbf{P}^1}  

\newcommand{\Hull}{\mathrm{Hull}}      
\newcommand{\Crit}{\mathrm{Crit}}     
\newcommand{\Ram}{\mathcal{R}}      
\newcommand{\PGL}{\mathrm{PGL}}    
\newcommand{\DD}{\mathcal{D}}     
\newcommand{\BB}{\mathcal{B}}    
\newcommand{\HH}{\mathbf{H}}     

\newcommand{\diam}{\mathrm{diam}}

\newcommand{\wronsk}{\mathrm{Wr}} 

\newtheorem*{fuzz*}{Fuzz Conjecture}
\newtheorem*{cic*}{Uniformity Conjecture}

\newtheorem*{thmD}{Theorem D}
\newtheorem*{thmE}{Theorem E}
\newtheorem*{thmF}{Theorem F}

\newtheorem{app}{Application}

\title
{Topology and Geometry of the Berkovich Ramification Locus for
	Rational Functions, II}

\author{Xander Faber \\
Department of Mathematics \\
University of Hawaii\\
Honolulu, HI  \\
xander@math.hawaii.edu
}

\date{}

\begin{document}

\maketitle

	\begin{abstract}
	    This article is the second installment in a series on the Berkovich ramification locus 
	    for nonconstant rational functions $\varphi \in k(z)$. 
	    Here we show the ramification locus is contained in a strong tubular neighborhood of finite radius
	    around the connected hull of the critical points of $\varphi$ if and only if $\varphi$ is tamely 
	    ramified. When the ground field $k$ has characteristic
	    zero, this bound may be chosen to depend only on the residue characteristic.  We give
	    two applications to classical non-Archimedean analysis, including a new version of 
	    Rolle's theorem for rational functions. 
	\end{abstract}

\section{Introduction}
    Let $k$ be any algebraically closed field that is complete with respect to a nontrivial non-Archimedean absolute value, and let $\varphi \in k(z)$ be a nonconstant rational function, viewed as an endomorphism of the Berkovich projective line $\Berk$. The ramification locus, denoted $\Ram_\varphi$, is the set of points of $\Berk$ at which the map $\varphi: \Berk \to \Berk$ is not locally injective. In \cite{Faber_Berk_RamI_2011}, we gave a description of the topological features of $\Ram_\varphi$. For example, it was proved that the ramification locus has at most $\deg(\varphi) - 1$ connected components \cite[Thm.~A]{Faber_Berk_RamI_2011}, that its interior points for the strong topology may be characterized by a certain local inseparability property \cite[Thm.~B]{Faber_Berk_RamI_2011}, and that it is well-behaved under extension of scalars \cite[\S4]{Faber_Berk_RamI_2011}. 
    
    Here we discuss geometric properties of the ramification locus that were hinted at in the introduction of the last article. The ``Berkovich hyperbolic space'' $\HH = \Berk \smallsetminus \PP^1(k)$ is preserved by the usual action of $\PGL_2(k)$, and there is a natural invariant metric $\rho$ on $\HH$ characterized by $\rho(\zeta_{a,r}, \zeta_{a,R}) = \log_{q_k}( R / r)$ whenever $a \in k$ and $0 < r < R$. See \cite[\S2.2.3]{Faber_Berk_RamI_2011} for further details. The morphism $\varphi$ is locally $\rho$-expanding on its ramification locus, and it is locally a $\rho$-isometry on the complement. So it is desirable to have some sort of control over the size of the ramification locus. Our two main theorems provide exactly that.
    
    For the statements, we need the following notation and terminology. If $X \subset \Berk$ is a nonempty subset and $r \geq 0$ is a real number, we define 
	\[
		X+r = \{y \in \Berk : \rho(x,y) \leq r \text{ for some $x \in X$}\}.
	\]
We will usually take $X = \Hull(\Crit(\varphi))$, the \textbf{connected hull} of the critical points of $\varphi$. (It is the smallest connected set containing all of the critical points.) 

	
\begin{thmD}[Uniform Tubular Neighborhood in Characteristic Zero]
\label{Thm: Uniformity}
	Suppose $k$ has characteristic zero and residue characteristic $p \geq 0$. Let $\varphi \in k(z)$ be a 
	nonconstant rational function. Then 
		\[
			\Ram_\varphi \subset \Hull(\Crit(\varphi)) + \begin{cases}
					0 & \text{if $p = 0$ or $p > \deg(\varphi)$} \\
					\frac{1}{p-1} & \text{if $0 < p \leq \deg(\varphi)$}.
				\end{cases}
		\] 
\end{thmD}

	After reading an earlier version of this article, Baldassarri investigated Theorem~D (and Application~1 below) using his methods on the radius of convergece of $p$-adic differential equations \cite{Baldassarri_convergence_2010, Baldassarri_Berk_Ram}. He defines a connection $(\varphi_* \OO_{\Berk}, \nabla)$ with the property that for any $y \in \Berk$, local sections of the morphism $\varphi: \Berk \to \Berk$ at $y$ form a $k$-basis of solutions to the system $(\varphi_* \OO_{\Berk}, \nabla)$. The ramification of $\varphi$ is precisely the obstruction to extending a given solution to a larger disk. While this technique does not recover the entirety of the theorem, Baldassarri is able to show that $\Ram_\varphi$ is contained in a tubular neighborhood of radius $1/(p-1)$ about $\Hull(Z)$, where $Z = \varphi^{-1}(\varphi(\Crit(\varphi)))$ is the saturation of $\Crit(\varphi)$ in $\Berk$. His result also applies in the more general context of finite morphisms of compact rig-smooth $k$-analytic curves. 

	Note that the bound for the tubular radius  in Theorem~D is independent of the rational function $\varphi$. The metric $\rho$ is normalized so that $\rho(\zeta_{0,1}, \zeta_{0, |p|^\alpha}) = \alpha$ for $\alpha \geq 0$, and the specific constant $1 / (p-1)$ is an artifact of this (natural) choice. When the characteristic of $k$ is positive, we are able to describe exactly when a bound of the above sort can exist. Recall that a rational function $\varphi \in k(z)$ is called \textbf{tamely ramified} if the characteristic of $k$ does not divide the multiplicity $m_\varphi(x)$ for any $x \in \PP^1(k)$.\footnote{Simply calling such rational functions ``tame'' is appealing, but Trucco has already reserved this term for a function whose ramification locus is contained in a finitely branched subtree of $\Berk$ \cite{Trucco_Tame_Polynomials_2012}.}

\begin{thmE}[Tubular Neighborhood in Positive Characteristic]
\label{Thm: Metric Properties}
	Suppose $k$ has positive characteristic, and let $\varphi \in k(z)$ be a nonconstant rational function. Then there exists a constant $r_\varphi \geq 0$ such that
		$
			\Ram_\varphi \subset \Hull(\Crit(\varphi)) + r_\varphi 
		$ 
if and only if $\varphi$ is tamely ramified.
\end{thmE}


	The ramification locus has finitely many connected components, each of which contains at least two critical points \cite[Thm.~A]{Faber_Berk_RamI_2011}. In light of Theorems~D and~E, the ramification locus may be viewed as lying inside a tubular neighborhood around a finite union of finitely branched trees, provided $\varphi$ is tamely ramified.

	



	As an application of Theorem~D, we can prove two interesting statements in classical non-Archimedean analysis. Write $D(a,r) = \{x \in k : |x - a| \leq r\}$ for the closed disk of radius $r \in \RR_{\geq 0}$ and center $a \in k$. If $p \geq 0$ is the residue characteristic of $k$ and $d \geq 1$ is an integer, we write $\gamma_{p,d} = 1$ if $p = 0$ or $p > d$ and $\gamma_{p,d} = |p|^{-1/ (p-1)} > 1$ otherwise.

\begin{app}[Non-Archimedean Rolle's Theorem for Rational Functions]
\label{App: Rolle's Theorem}
	Suppose $k$ has characteristic zero and residue characteristic $p \geq 0$. Let $\varphi \in k(z)$ be a rational function of degree $d\geq 1$. If $\varphi$ has two distinct zeros in the closed disk $D(a,r)$, then it has a critical point in the closed disk $D(a, r\cdot \gamma_{p,d})$.
\end{app}

    Robert proved a version of this result for power series that converge on the disk $D(a, r\cdot \gamma_{p,d})$ \cite[\S2.4]{Robert_p-adic_Book_2000}. A weak form of it is a key ingredient in our argument. 
In basic calculus, Rolle's Theorem is intimately connected with the mean value theorem. But while the literature contains several variations on the mean value theorem for non-Archimedean analytic functions (e.g., \cite{Rivera-Letelier_Injective_2002} or \cite[IV.3]{Robert_p-adic_Book_2000}),  it does not appear to be possible to formulate one in the present context due to the existence of rational functions that map a closed disk $D(a,r)$ over the entire projective line. Our second application asserts that there is a critical point nearby if this occurs.   

\begin{app}
\label{App: Exploding Disk}
	Suppose $k$ has characteristic zero and residue characteristic $p \geq 0$. Let $\varphi \in k(z)$ be a rational function of degree $d\geq 1$. If $\varphi(D(a,r)) = \PP^1(k)$, then there exists a critical point in the disk $D(a,  r \cdot\gamma_{p,d})$. 
\end{app}


    To prove all of the above results, we explicitly construct a function $\tau_\varphi: \HH \to \RR$ called the \textbf{visible ramification}. It is continuous for the strong topology, nonnegative, and it has the following important feature: for each $x \in \HH$ and each tangent direction $\vec{v} \in T_x$ such that $\BB_x(\vec{v})^-$ does not contain a critical point, we have 
    \[
    	\max_{\zeta \in \BB_x(\vec{v})^- \cap \Ram_\varphi} \rho(x, \zeta) \leq \tau_\varphi(x),
    \]
with equality for all but finitely many directions $\vec{v}$. Here we interpret the maximum to be zero if the ramification locus does not intersect $\BB_x(\vec{v})^-$. Theorem~D is a consequence of our study of the visible ramification and Robert's non-Archimedean Rolle's Theorem for power series. Theorem~E follows by a different argument once we know that $\tau_\varphi$ remains bounded in a weak neighborhood of each tame critical point. In fact, we completely determine the shape of the ramification locus near any critical point:

\begin{thmF}
\label{Thm: Local Fuzz}
	Let $k$ be a field with residue characteristic $p \geq 0$, let $\varphi \in k(z)$ be a nonconstant separable rational function, and let $c$ be a critical point of $\varphi$ with multiplicity $m = m_\varphi(c) > 1$. There exists an open Berkovich disk $U = U(\varphi, c)$ about $c$ such that
	\[
		U \cap \Ram_\varphi = U \cap \left[ \Hull(\Crit(\varphi)) + 
			\begin{cases} 
				0 & \text{if $p \nmid m$} \\
				\frac{1}{p-1} & \text{if $\mathrm{char}(k) = 0$ and $p \mid m$}
			\end{cases}\right].
	\]
If $\mathrm{char}(k) = p$ and $p \mid m$, then there exists no open Berkovich disk $U$ centered at $c$ such that $U \cap \Ram_\varphi$ is contained in a strong closed tube about $U \cap \Hull(\Crit(\varphi))$ with finite radius.
\end{thmF}

	

    In Section~\ref{Sec: Aux Poly} we construct a two-variable auxiliary polynomial $A_\varphi(z,w)$ that essentially carries the same information as the Taylor series of $\varphi(z+w) - \varphi(w)$ centered at $z = 0$. We use this polynomial to define the visible ramification $\tau_\varphi$ and prove its key properties in Section~\ref{Sec: Visible}. This is accomplished via a series of lemmas that connect the algebraic data of $A_\varphi$ to the desired geometric interpretation of $\tau_\varphi$. In Section~\ref{Sec: Alternate} we a slightly different definition of the visible ramification that can be useful for computing the ramification locus. Section~\ref{Sec: Proofs} contains the proofs of Theorem~D and the two applications. Section~\ref{Sec: Local Fuzz} is reserved for computing the behavior of the visible ramification near a critical point. This amounts to an elaborate calculation involving comparisons of binomial coefficients. The proof of Theorem~E also appears there. We conclude with three examples in Section~\ref{Sec: Examples} that illustrate, among other things, that Theorems~D and~E cannot be improved.   
    
    Finally, we note that this article is a direct sequel to \cite{Faber_Berk_RamI_2011}, and we expect it to be read as a companion piece. In particular, we have endeavored to make Sections~2 and~3 of the prequel a rather encyclopedic reference for all of the necessary foundational facts on the Berkovich projective line and on multiplicities for rational functions. To keep the reader's work to a minimum, we will usually cite \cite{Faber_Berk_RamI_2011} when referring to these results; proper credit is assigned there.


\section{Visible Ramification}
\label{Sec: Auxiliary}

\noindent    \textbf{Convention 1.} Throughout this section, we assume that $\varphi \in k(z)$ is a separable nonconstant rational function with at least two critical points. (If $\varphi$ has only one critical point, then $k$ has positive characteristic and the unique critical point of $\varphi$ is wildly ramified.) \\
    
\noindent    \textbf{Convention 2.} We further assume that $\infty$ is a critical point of $\varphi$. Although it is not strictly necessary for what follows, it will simplify some of the discussion.  \\


\subsection{An Auxiliary Polynomial}
\label{Sec: Aux Poly}

	Let $\varphi \in k(z)$ be a nonconstant rational function, and let us write it as $\varphi = f / g$ for some polynomials $f, g \in k[z]$ with no common root. Recall that the \textbf{Wronskian} of $\varphi = f / g$ is defined as
	\[
		\wronsk_\varphi(z) = f'(z) g(z) - f(z) g'(z).
	\]
The roots of $\wronsk_\varphi$ are precisely the finite critical points of $\varphi$.

	For $\varphi = f /g$ as above, define a 2-variable rational function $A_\varphi(z,w)$ by the formula
	\[
		A_\varphi(z, w) = \frac{1}{z}\left[f(z+w)g(w) - f(w)g(z+w)\right].
	\]
We have suppressed the dependence of $A_\varphi$ on $f$ and $g$ from the notation. It is clearly a polynomial in the variable~$w$, and we will show that momentarily that it is also a polynomial in $z$. 

	As a motivating example, consider the case where $\varphi(z) = f(z)$ is a polynomial and $\mathrm{char}(k) = 0$. Viewing $w$ as a fixed parameter, we may expand $f(z+w)$ as a series about $z = 0$ to find that 
	\[
		f(z+w) - f(w) = f'(w) \, z + \frac{f''(w)}{2!} \, z^2 + \cdots + \frac{f^{(d)}(w)}{d!} \, z^d.
	\]
Hence $A_\varphi(z,w)$ is a polynomial in $z$ of degree $d-1$. Evidently $A_\varphi(0,w) = f'(w)$ is the Wronskian. Moreover, if $A_\varphi(z,w)$ does not vanish on the disk $D(0,r)$, then $\varphi$ is injective on the disk $D(w,r)$. These last two observations allow us to connect $A_\varphi$ with the ramification of $\varphi$. All of these properties will persist, at least in part, to the more general case of rational functions. 

\begin{prop}
	Let $\varphi = f / g \in k(z)$ be a nonconstant rational function of degree~$d$, and define the function $A_\varphi(z,w)$ as above. Write $f(z) = a_dz^d + \cdots + a_0$ and $g(z) = b_dz^d + \cdots + b_0$  with $a_i, b_j \in k$. Then $A_\varphi$ is given explicitly by $A_\varphi(z,w) = \sum_{\ell \geq 0} A_{\ell}(w) z^\ell$, where
	\[
		A_\ell(w) = \sum_{0 \leq i, j \leq d}  \left[ \binom{i}{\ell + 1} - \binom{j}{\ell+ 1} \right]
				a_i b_j w^{i+j - \ell - 1}.	
	\]
Moreover, the following statements are true:
	\begin{enumerate}
		\item $A_\varphi(z,w)$ has degree $d - 1$ as a polynomial in $z$.
		\item $A_0(w) = \wronsk_\varphi(w)$. In particular, $A_0(w)$ is identically zero if and 
			only if $\varphi$ is inseparable. 
	\end{enumerate}	
\end{prop}


\begin{proof}
	To ease notation, let us make the convention that $a_i = b_i = 0$ if $i < 0$ or $i > d$. Then
	\benn
	\ba
		z \cdot A_\varphi(z,w) &= 
			\sum_{i \geq 0} a_i(z+w)^i \sum_{j \geq 0} b_j w^j - \sum_{i \geq 0} a_i w^i \sum_{j \geq 0}
			b_j(z+w)^j  \\
			&=  \sum_{i,j \geq 0} a_i b_j \sum_{0 \leq \ell \leq i} \binom{i}{\ell} w^{i+j - \ell} z^{\ell}
			-\sum_{i,j \geq 0} a_i b_j \sum_{0 \leq \ell \leq j} \binom{j}{\ell} w^{i+j - \ell}z^{\ell} \\
			&=  \sum_{\ell \geq 0} \left\{ \sum_{\substack{i \geq \ell \\ j \geq 0}} a_i b_j
				\binom{i}{\ell} w^{i+j - \ell} - \sum_{\substack{i \geq 0 \\ j \geq \ell}} a_i b_j
				\binom{j}{\ell} w^{i+j - \ell} \right\}z^{\ell} \\
			&= z  \cdot \sum_{\ell \geq 1} \left\{ \sum_{i,j \geq 0} \left[ \binom{i}{\ell} - \binom{j}{\ell} \right]
				a_i b_j w^{i+j - \ell} \right\}z^{\ell - 1},
	\ea
	\eenn
where the last line follows from the previous one upon observing that the $\ell = 0$ term vanishes (trivially) and $\binom{i}{\ell} = 0$ for $0 \leq i \leq \ell - 1$ and $\ell \geq 1$. Replacing $\ell$ with $\ell + 1$ and dividing by $z$ gives the desired explicit formula for $A_\varphi(z,w)$. 

	Since  $\binom{i}{\ell + 1} = 0$ if $0 \leq i \leq \ell$,  it follows that when $\ell \geq d$, we have 
	\[
		\left[ \binom{i}{\ell + 1} - \binom{j}{\ell+ 1} \right] a_i b_j = 0 \quad \text{for all $i, j \geq 0$}.
	\]
Hence $A_\ell(w) = 0$ for all $\ell \geq d$. 

	We now show that $A_{d - 1} \neq 0$. Observe that $\binom{i}{d} = \binom{j}{d}$ if both $i \leq d-1$ and $j \leq d - 1$, or if $i = j = d$. Hence
	\benn
		\ba
			A_{d-1}(w) &= \sum_{i,j \geq 0} \left[ \binom{i}{d} - \binom{j}{d} \right]
				a_i b_j w^{i+j - d} \\
			 &= - \sum_{0 \leq i \leq d-1} a_i b_d w^i + \sum_{0 \leq j \leq d-1} a_d b_j w^j \\
			&= \sum_{0 \leq i \leq d-1} (a_d b_i - a_i b_d) w^i.
		\ea
	\eenn
If $A_{d-1} = 0$, we find $a_db_i = a_i b_d$ for all $i = 0, \ldots, d-1$. If also $a_d \neq 0$, then $b_i = a_ib_d / a_d$ for $i = 0, \ldots, d-1$, which implies that $\varphi(z) = a_d / b_d$ is constant, a contradiction. A similar argument applies if $b_d \neq 0$. Since $\varphi$ has degree~$d$, at least one of $a_d$ or $b_d$ is nonzero, and hence we have arrived at a contradiction. 

	Finally, to see that that $A_0(w)$ agrees with the Wronskian of $\varphi$, we compute:
	\benn
		\ba
		\wronsk_\varphi(w) &= f'(w)g(w) - f(w)g'(w) \\
			&= \sum_{i, j \geq 0}i a_i b_jw^{i+j-1} - \sum_{i,j \geq 0} ja_i b_j w^{i+j - 1} \\
			&= \sum_{i, j \geq 0} \left[\binom{i}{1} - \binom{j}{1} \right]a_i b_j w^{i+j - 1} = A_0(w).\\
		\ea
	\eenn
As $\wronsk_\varphi$ is the numerator of the derivative $\varphi'$, the equivalence between inseparability and vanishing of the Wronskian is \cite[Prop.~2.3]{Faber_Berk_RamI_2011}.
\end{proof}


\subsection{Definition and Properties of the Visible Ramification}
\label{Sec: Visible}

    Define the \textbf{visible ramification} $\tau_\varphi: \HH \to \RR$ by the formula
        \[
            \tau_\varphi(x) = \begin{cases}
            	\displaystyle \log_{q_k} \diam(x) 
			+ \max_{\ell \geq 1} \log_{q_k}\left| \frac{A_\ell}{A_0}(x) \right|^{1/\ell}
		& \text{if $x \in \Ram_\varphi$} \\
		0 & \text{if $x \not\in \Ram_\varphi$}.
		\end{cases}
        \]
Here $\diam(\cdot)$ is the affine diameter defined in \cite[\S2.2.3]{Faber_Berk_RamI_2011}; for example, if $x = \zeta_{a,r}$ is a point of type~II or~III, then $\diam(x) = r$. Note that $A_0$ and $A_{d-1}$ are not identically zero, so $\tau_\varphi(x)$ is well-defined. The proofs of the next two propositions will occupy the remainder of this section.\footnote{The visible ramification is related to Baldassarri's notion of generic radius of convergence \cite[Def.~2.0.9]{Baldassarri_Berk_Ram} via the formula 
$\tau_\varphi(x) = -\log_{q_k} R_{D \subset \Berk} (\Sigma)$. Here $\Sigma$ is the first order system of linear differential equations associated to the connection $(\varphi_*\OO_{\Berk}, \nabla)$ alluded to in the discussion after Theorem~D.}

\begin{prop}
\label{Prop: Tau properties}
	The visible ramification $\tau_\varphi$ enjoys the following four properties:
	
	\begin{enumerate}
		\item\label{Item: Strong Continuous}
			 $\tau_\varphi$ is nonnegative and continuous for the strong topology on $\HH$. 
		
		\item\label{Item: Finite Subgraph}
			 If $\Gamma$ is any finite subgraph of $\HH$, then $\tau_\varphi|_\Gamma$ is 
			piecewise affine with respect to a parameterization using the path-metric distance $\rho$. 
			
		\item\label{Item: Vanishing} $\tau_\varphi(x) = 0$ if and only if $\varphi$ has separable reduction at $x$.
		  
		  \item\label{Item: Support} The topological support of $\tau_\varphi$ agrees with the set of 
		  	points $x \in \HH$ such that the residue characteristic of $k$ divides $m_\varphi(x,\vec{v})$ 
			for some tangent direction $\vec{v} \in T_x$.\footnote{We could rephrase this condition by
			saying that $\varphi$ has \textbf{wildly ramified reduction} at $x$.}
	
	\end{enumerate}
\end{prop}


    Given $x \in \HH$, define two exceptional sets of tangent directions $E_{\varphi, x} \subset E_{\varphi, x}' \subset T_x$ as follows: 
    \begin{eqnarray*}
        E_{\varphi, x} &=&  \{\vec{v} \in T_x : \BB_x(\vec{v})^- \cap \Crit(\varphi) \neq \emptyset \} \\
        E_{\varphi, x}' &=& \{\vec{v} \in T_x : \BB_x(\vec{v})^- \text{ contains $\infty$ or a root of some nonzero } A_\ell(w), 
        \ell = 0, \ldots, d-1\}.
    \end{eqnarray*}
As $\varphi$ is separable, these sets are finite. It is impossible to control the ramification in the directions  lying in $E_{\varphi, x}$ because of the presence of critical points. Each of the directions in $T_x \smallsetminus E_{\varphi, x}'$ looks the same with regard to ramification and will be relatively easy to deal with, while the ramification in the directions $E_{\varphi, x}'$ is much more subtle.

\begin{prop}
\label{Prop: Ramification Bounded by Tau}
	Let $x \in \HH$ be a type~II point, let $\vec{v} \in T_x \smallsetminus E_{\varphi,x}$, and write $U = \BB_x(\vec{v})^-$. Then 
	\[
		\max_{\zeta \in U \cap \Ram_\varphi} \rho(x, \zeta) \leq \tau_\varphi(x),
	\]
with equality if $\vec{v} \not\in E_{\varphi,x}'$. (Here we interpret the maximum as zero if $U \cap \Ram_\varphi$ is empty.)
\end{prop}

\begin{remark}
	Equality need not be attained in the proposition for $\vec{v} \in E_{\varphi, x}' \smallsetminus E_{\varphi, x}$. 
	See Example~\ref{Ex: Pathology}.
\end{remark}

	The strategy of the proofs of these results is as follows. Let $x = \zeta_{a, R}$ be a type~II point in the ramification locus. Given a classical point $y \prec x$ that is not a critical point, we will associate two quantities to it. The first is $r_y$: the minimum radius of a disk such that $\zeta_{y, r_y} \in \Ram_\varphi$.  The second is $r_y'$: the minimum absolute value of a root of the polynomial $A_\varphi(z, y)$. On one hand, the quantity $r_y$ is geometric, and calculating it is equivalent to determining the maximum distance that ramified points can extend from $x$ toward~$y$. So it will be necessary to bound $r_y$ below to prove Theorems~D and~E. On the other hand, the quantity $r_y'$ is algebraic and the Newton polygon can be used to give bounds for it. Generically, we expect that $r_y = r_y'$ because ramification and vanishing of $A_\varphi(z,y)$ both signal many-to-one behavior of the function~$\varphi$. Most of our effort will be devoted to the cases where this equality fails. The following lemma is a first step in justifying all of these claims. 
	
\begin{lem}
\label{Lem: Ramification Segment}
	Let $\varphi \in k(z)$ be a separable nonconstant rational function, and let $U = \DD(a, R)^-$ 
	be an open Berkovich disk. Suppose that $U$ contains no critical point of $\varphi$, and that 
	$\overline{U} \cap \Ram_\varphi$ is nonempty. For each classical point $y \in U$, 
	there exists a real number $r_y$ satisfying $0 < r_y \leq R$ such that 
	\[
	    [y, \zeta_{a,R}] \cap \Ram_\varphi = [\zeta_{y, r_y}, \zeta_{a,R}].	
	\]
    If $\varphi(\DD(y, r_y)) \neq \Berk$, then $r_y$ is the minimum absolute value 
    of a root of $A_\varphi(z,y)$. 
\end{lem}

\begin{proof}
    Write $Z = [y, \zeta_{a,R}] \cap \Ram_\varphi$. If $U \cap \Ram_\varphi$ is empty, then we must have $Z = \partial U = \{\zeta_{a,R}\}$. Set $r_y = R$. 
    
    If instead $U \cap \Ram_\varphi$ is nonempty, then $\varphi$ has inseparable reduction at $\zeta_{a,R}$ \cite[Prop.~6.6]{Faber_Berk_RamI_2011}, so that this point lies in the ramification locus. In particular, $Z$ is nonempty. Define $r_y$ to be the infimum of the set $\{ r :  0\leq r \leq R, \zeta_{y,r} \in \Ram_\varphi\}$. Since $\Ram_\varphi$ is closed, the infimum is attained, so that $\zeta_{y, r_y} \in \Ram_\varphi$. The whole segment $[\zeta_{y,r_y}, \zeta_{a,R}]$ lies inside the ramification locus; for otherwise $U$ would contain an entire connected component of $\Ram_\varphi$, which would contradict the fact that each connected component of the ramification locus contains a critical point \cite[Thm.~A]{Faber_Berk_RamI_2011}. Observe that $r_y > 0$ because $U$ contains no critical point. 
    
    Before we complete the proof, we observe that if $\sigma \in \PGL_2(k)$, then the roots of $A_\varphi(z,y)$ and $A_{\sigma \circ \varphi}(z,y)$ agree. Indeed, if $\varphi = f / g$ and $\sigma(z) = (\alpha z + \beta) / (\gamma z + \delta)$, then $\sigma \circ \varphi = (\alpha f + \beta g) / (\gamma f + \delta g)$, and a direct computation shows that 
    \[
        A_{\sigma \circ \varphi}(z,w) = (\alpha \delta - \beta \gamma) A_\varphi(z,w). 
    \]
Moreover, the ramification loci for $\varphi$ and $\sigma \circ \varphi$ agree \cite[Cor.~3.7]{Faber_Berk_RamI_2011}.
    
    Now suppose that $\varphi(\DD(y, r_y)) \neq \Berk$. Without loss of generality, we may assume that $\infty$ does not lie in the image of $\DD(y,r_y)$ under $\varphi$. Indeed, if $x \neq \infty$ is a classical point that does not lie in the image of $\DD(y, r_y)$, then let $\sigma(z) = 1 / (z - x)$. Now $\infty \not\in \sigma \circ \varphi \left(\DD(y, r_y)\right)$. The previous paragraph shows that it suffices to replace $\varphi$ with $\sigma \circ \varphi$ in order to prove the lemma.
    
    By construction, we have $m_\varphi(\zeta_{y, r_y}) > 1$ and $m_\varphi(\zeta_{y, r}) = 1$ for all $r < r_y$. A type~III point is never an endpoint of the ramification locus \cite[Prop.~7.3]{Faber_Berk_RamI_2011}, so $\zeta_{y, r_y}$ must be of type~II. Choose $t \in k$ such that $0 < |t| \leq r_y$. Set $\psi(z) = \varphi(tz+y) - \varphi(y)$. Then $\psi(0) = 0$ and $\psi$ has no pole in the disk $D(0,1)$. The following lemma implies that $\psi$ has a nonzero root in the disk $D(0, 1)$ if and only if $|t| = r_y$, and so $\varphi(z+y) = \varphi(y)$ has a nonzero solution in the disk $D(0, |t|)$ if and only if $|t| = r_y$. Since $A_\varphi(z, y) = z^{-1}g(y)g(z+y)[\varphi(z+y) - \varphi(y)]$, and since $g(z+y)$ does not vanish for $z$ in the disk $D(0, r_y)$, we conclude that $r_y$ is the minimum absolute value of a root of $A_\varphi(z,y)$.
\end{proof}

\begin{lem}
\label{Lem: Multiplicity and zeros}
	Let $\psi(z) \in k(z)$ be a nonconstant separable rational function with no pole in the disk $D(0,1)$ such that $\psi(0) = 0$. For each real number $0 < r \leq 1$, the multiplicity satisfies 
	\[
		m_\psi(\zeta_{0,r}) = \#\{\text{zeros of $\psi$ in $D(0,r)$}\}.  
	\]
\end{lem}

\begin{proof}
	The poles of a nonconstant rational function are discrete, and so $\psi$ has no pole in $D(0, 1 + \varepsilon)$ for some $\varepsilon > 0$. Consider $r'$ such that $0 < r < r' <1 + \varepsilon$. The hypotheses of the lemma guarantee that $\psi(D(0,r')^-)$ is an open disk $D(0,s')^-$ for some $s' \in \RR_{> 0}$, and that the induced map $D(0,r')^- \to D(0,s')^-$ is $m$-to-$1$ for some $m = m(r') \geq 1$.  In particular, $m$ is the number of zeros of $\psi$ in the disk $D(0,r')^-$.   On the other hand, we know that $m = m_\psi(\zeta_{0,r'}, \vec{0})$ \cite[Prop.~3.10]{Faber_Berk_RamI_2011}. But the zeros of $\psi$ are discrete, and so the number of zeros in $D(0,r')^-$ agrees with the number in $D(0,r)$ if $r'$ is sufficiently close to $r$. In particular, $m = m_\psi(\zeta_{0,r'}, \vec{0})$ is stable for $r'$ sufficiently close to $r$.  To complete the proof we must show that $m_\psi(\zeta_{0,r}) = m_\psi(\zeta_{0,r'}, \vec{0})$ if $r'$ is sufficiently close to $r$. 
	
	Observe that $m_\psi(\zeta_{0,r}) \leq m_\psi(\zeta_{0,r'})$  \cite[Prop.~3.11]{Faber_Berk_RamI_2011}, and that upper semicontinuity of the multiplicity function $m_\psi$ guarantees that
	\[
		\lim_{r' \to r} m_\psi(\zeta_{0,r'}) \leq m_\psi(\zeta_{0,r}).
	\]
Since the multiplicity may take on only a finite set of values, we find that $m_\psi(\zeta_{0,r}) = m_\psi(\zeta_{0,r'})$ for all $r'$ sufficiently close to $r$. Let us fix $r'$ with this property. 
	
	There exists $r''$ such that $r < r'' < r'$ and $m_\psi(\zeta_{0,r'}, \vec{0}) = m_\psi(\zeta_{0,r''})$ \cite[Prop.~3.9]{Faber_Berk_RamI_2011}. The previous paragraph shows that $m_\psi(\zeta_{0,r''}) = m_\psi(\zeta_{0,r})$, and so the proof is finished. 
\end{proof}

\begin{lem}
\label{Lem: Small Root Bound}
    Let $x$ be a type~II point, let $\vec{v} \in T_x \smallsetminus E_{\varphi, x}$, and write $U = \BB_x(\vec{v})^-$.
    Suppose $y \in U$ is a classical point, and write $r_y'$ for the minimum absolute value of a root of $A_\varphi(z,y)$;
    interpret $r_y'$ as $+\infty$ if $A_\varphi(z,y)$ is a constant polynomial. Then
    \be
    \label{Eq: Small Root Bound}
        r_y' \geq \left( \max_{\ell \geq 1} \left| \frac{A_\ell}{A_0}(x) \right|^{1/\ell} \right)^{-1}.
    \ee
If $\vec{v} \not\in E_{\varphi, x}'$, then equality holds in \eqref{Eq: Small Root Bound}.  
\end{lem}

\begin{remark}
	Since the right side of \eqref{Eq: Small Root Bound} is finite, the final statement of the lemma implies that $r_y' = +\infty$ can only occur if $y \in \BB_x(\vec{v})^-$ with $\vec{v} \in E_{\varphi, x}'$. 
\end{remark}

\begin{proof}
      Write $c_1, \ldots, c_n$ for the distinct finite critical points of $\varphi$. As $\vec{v} \in T_x \smallsetminus E_{\varphi, x}$, no $c_i$ lies in~$U$. Let $x = \zeta_{a,r}$. Now $\vec{v}$ cannot point toward $\infty$ since the latter is a critical point; so $U = \BB_x(\vec{v})^- \subset \DD(a,r)$. If $|c_i - a|\leq r$, then $|y - c_i| = r$ (else $\vec{v} \in E_{\varphi,x}$). But if $|c_i - a| > r$, then $|y - c_i| = |c_i - a|$. Write $\wronsk_\varphi(w) = C \prod (z - c_i)^{w_\varphi(c_i)}$ for some nonzero $C \in k$. Then
        \benn
            \ba
            |A_0(y)| = |\wronsk_\varphi(y)| &= |C| \prod_i |y - c_i|^{w_\varphi(c_i)} \\
                &= |C| \prod_{i \ : \ |c_i - a| \leq r} r^{w_\varphi(c_i)} \cdot 
                    \prod_{i \ : \ |c_i - a| > r} |c_i - a|^{w_\varphi(c_i)}.
            \ea
        \eenn
This last expression is independent of the choice of $\vec{v} \in T_x \smallsetminus E_{\varphi, x}$ and $y \in \BB_x(\vec{v})^-$. So the first vertex of the Newton polygon of $A_\varphi(z, y)$ (as a polynomial in $z$) is also independent of these choices.

    An argument similar to the one in the last paragraph shows that all other vertices of the Newton polygon of $A_\varphi(z,y)$ are stable as one varies $\vec{v} \in T_x \smallsetminus E_{\varphi, x}'$ and $y \in \BB_x(\vec{v})^-$. In particular, if $\mu$ is the slope of the first face of the resulting Newton polygon, then $\mu$ does not depend on these choices. The Newton polygon is a lower convex hull, so
    \benn
        \ba
        \mu &= \min_{\ell \geq 1} \frac{\ord_k (A_\ell(y)) - \ord_k(A_0(y))}{\ell} \\
            &= - \max_{\ell \geq 1} \ \log_{q_k} \left| \frac{A_\ell(y)}{A_0(y)} \right|^{1 / \ell} \\
            &= - \log_{q_k} \max_{\ell \geq 1} \left| \frac{A_\ell}{A_0}(x) \right|^{1 / \ell}.
        \ea
    \eenn
The last line follows by continuity and the fact that $\{x\} = \partial U$ is a limit point of such $y$ for the weak topology. We conclude that for any $\vec{v} \in T_x \smallsetminus E_{\varphi, x}'$ and any $y \in \BB_x(\vec{v})^-$, 
    \[
        r_y' = q_k^\mu = \left( \max_{\ell \geq 1} \left| \frac{A_\ell}{A_0}(x) \right|^{1/\ell} \right)^{-1}. 
    \]

    Now let $\vec{v}_0 \in T_x \smallsetminus E_{\varphi, x}'$ and $y_0 \in \BB_x(\vec{v}_0)^-$, which we view as generic choices of these parameters. 
Observe that if $\vec{v} \in E_{\varphi, x}' \smallsetminus E_{\varphi, x}$, and if $y \in \BB_x(\vec{v})^-$, then 
     \begin{eqnarray*}
        |A_\ell(y)| &\leq& |A_{\ell}(y_0)|,\  \ell = 1, \ldots, d-1 \\
        |A_0(y)| & = & |A_0(y_0)|.
     \end{eqnarray*}
Indeed, $y$ is allowed to be closer to a root of $A_\ell$ than $y_0$ is. So the first face of the Newton polygon of $A_\varphi(z,y)$ must lie on or above that of $A_\varphi(z,y_0)$, while they share the same first vertex. In particular, the smallest root of $A_\varphi(z,y)$ (if it exists) has absolute value at least that of $A_\varphi(z,y_0)$, and the calculation in the previous paragraph gives \eqref{Eq: Small Root Bound}.
\end{proof}

\begin{lem}
\label{Lem: Diameter bound}
	Let $x \in \Ram_\varphi$ be a type~II point, let $\vec{v} \in T_x \smallsetminus E_{\varphi, x}'$, and let $y \in \BB_x(\vec{v})^-$ be a classical point. Then the smallest root of $A_\varphi(z,y)$ has absolute value at most $\diam(x)$, with equality precisely when $\varphi$ has separable reduction at $x$. 
\end{lem}

\begin{proof}
	If $\varphi(y) = \infty$, then we replace $\varphi$ with $1 / \varphi(z)$ so that $\varphi(y) = 0$. As in the proof of Lemma~\ref{Lem: Ramification Segment}, this change of coordinate multiplies $A_\varphi$ by a nonzero constant (in this case~$-1$), and it has no effect on the ramification locus of $\varphi$. 
	
	We first consider the case in which $\varphi$ has inseparable reduction at $x$. Then $x$ is a strong interior point of the ramification locus \cite[Prop.~7.4]{Faber_Berk_RamI_2011}, so that $r_y < \diam(x)$. Here $r_y$ is defined as in the statement of Lemma~\ref{Lem: Ramification Segment}. Write $x = \zeta_{a,R}$, where $R = \diam(x)$. Choose $1 + \varepsilon \in |k|^\times$ with $\varepsilon > 0$ small enough that $r_y(1 + \varepsilon) < R$, and let $\xi = \zeta_{y, r_y (1+ \varepsilon)}$. Then $\varphi$ has inseparable reduction at $\xi$ \cite[Prop.~6.6]{Faber_Berk_RamI_2011}. In particular, there are infinitely many directions $\vec{w} \in T_\xi$ such that $m_\varphi(\xi, \vec{w}) > 1$. Choose one satisfying $\varphi(\BB_{\xi}(\vec{w})^-) \neq \Berk$, and let $y_1 \in \BB_\xi(\vec{w})^-$. Then $\varphi(\DD(y_1, r_{y_1})) \neq \Berk$, so that
	\benn
		\ba
			r_y' &= r_{y_1}'  \qquad (\text{Lemma}~\ref{Lem: Small Root Bound}) \\
				&=  r_{y_1} \qquad (\text{Lemma}~\ref{Lem: Ramification Segment}) \\
				& < \diam(x).
		\ea
	\eenn
The proof is complete in this case.

	Now suppose that $\varphi$ has separable reduction at $x$. We know that $m_\varphi(x) > 1$ by hypothesis. So the map $\varphi_*: T_x \to T_{\varphi(x)}$ is surjective with degree at least~2 \cite[Prop.~3.14]{Faber_Berk_RamI_2011}. Since the absolute value of the smallest root $r_y'$ is independent of the choice of $\vec{v} \in T_x \smallsetminus E_{\varphi, x}'$ and $y \in \BB_x(\vec{v})^-$ (Lemma~\ref{Lem: Small Root Bound}), we are free to make other choices in order to simplify the proof. After choosing another $\vec{v}$ if necessary, we may select a second tangent direction $\vec{w} \in T_x \smallsetminus E_{\varphi, x}'$ such that $\varphi_*(\vec{v}) = \varphi_*(\vec{w})$ and so that no pole lies in $\BB_x(\vec{v})^-$ or $\BB_x(\vec{w})^-$. In particular, for $y \in \BB_x(\vec{v})^-$, there is $y_1 \in \BB_x(\vec{w})^-$ with $\varphi(y) = \varphi(y_1)$. 
	
	 Write $x = \zeta_{a,R}$ as above. Since $y, y_1$ lie in distinct directions at $x$, we find that $|y_1 - y| = R$. Then
	\[
		A_\varphi(y_1 - y,y) = (y_1 - y)^{-1}\left[f(y_1) g(y) - f(y)g(y_1)\right] = 0.
	\]
Hence $y_1 - y$ is a root of $A_\varphi(z, y)$. Let $\alpha$ be a root with minimum absolute value. Then we have just proved that $|\alpha| = r_y' \leq R = \diam(x)$.  It remains to show that equality holds. 

	If $|\alpha| < R$, define $y_2 = y + \alpha$, so that $y_2 \in \DD(y, R)^- = \BB_x(\vec{v})^-$. Then
	\benn
		0 = A_\varphi(\alpha, y) = \alpha^{-1} \left[f(\alpha + y) g(y) - f(y)g(\alpha + y)\right] 
			= \alpha^{-1} g(y) g(y_2)\left[ \varphi(y_2) - \varphi(y)\right].
	\eenn
Note that since $\BB_x(\vec{v})^-$ contains no pole of $\varphi$, the denominators $g(y)$ and $g(y_2)$ do not vanish. Hence $\varphi(y_2) = \varphi(y)$, which means $\varphi$ is not injective on the open Berkovich disk $\BB_x(\vec{v})^-$, and hence $\BB_x(\vec{v})^- \cap \Ram_\varphi$ is nonempty \cite[Cor.~3.8]{Faber_Berk_RamI_2011}. As $\BB_x(\vec{v})^-$ contains no critical point, we find that $\varphi$ has inseparable reduction at $x$ \cite[Prop.~6.6]{Faber_Berk_RamI_2011}, which is a contradiction.  
We are forced to conclude that $r_y' = |\alpha| = R$. 
\end{proof}

\begin{proof}[Proof of Proposition~\ref{Prop: Tau properties}]
	The visible ramification $\tau_\varphi$ is identically zero on the strong open set $\HH \smallsetminus \Ram_\varphi$, so it is trivially continuous there. Note also that $\varphi$ has separable reduction at any $x \in \HH \smallsetminus \Ram_\varphi$ since $m_\varphi(x) = 1$. 
	
	Evidently the affine diameter is continuous and non-vanishing on $\HH$ for the strong topology, so that $\log_{q_k} \diam(x)$ is a continuous function of $x$. When restricted to any finite subgraph $\Gamma$, it is also piecewise affine. Each nonzero rational function $A_\ell / A_0$ gives rise to a function $\log_{q_k} \left| A_\ell / A_0 \right|$ that is harmonic with respect to the Laplacian operator $\Delta$ away from the roots of $A_0$ and $A_\ell$ (which do not lie in $\HH$); see \cite[Ch.3, \S3.2]{AWS_2008}. The functions $\log_{q_k}\left| A_\ell / A_0 \right|^{1/\ell}$ are continuous for the strong topology on $\HH$ and piecewise affine on any finite subgraph $\Gamma \subset \HH$,  and  so their maximum $(1 \leq \ell \leq d-1$) must be as well. This completes the proof of~\eqref{Item: Finite Subgraph} and the continuity part of~\eqref{Item: Strong Continuous}. 
	
	Now let $x \in \Ram_\varphi$ be a type~II point, let $\vec{v} \in T_x \smallsetminus E_{\varphi,x}'$ be a tangent direction, and let $y \in \BB_x(\vec{v})^-$ be an arbitrary classical point. Write $r_y'$ for the minimum absolute value of a root of $A_\varphi(z,y)$. Then Lemmas~\ref{Lem: Small Root Bound} and~\ref{Lem: Diameter bound} show that
	\[
		 \left( \max_{\ell \geq 1} \left| \frac{A_\ell}{A_0}(x) \right|^{1/\ell} \right)^{-1} = r_y' \leq \diam(x).
	\]
Dividing by $\diam(x)$ and applying $- \log_{q_k}$ yields
	\be
	\label{Eq: tau inequalities}
		0 \leq \log_{q_k} \frac{\diam(x)}{r_y'} = \tau_\varphi(x). 
	\ee
Hence $\tau_\varphi$ is nonnegative at all type~II points of $\Ram_\varphi$. By continuity it is nonnegative on the entire ramification locus, and the proof of~\eqref{Item: Strong Continuous} is complete. 
	
	Keep the notation from the last paragraph. We saw in Lemma~\ref{Lem: Diameter bound} that $r_y' = \diam(x)$ if and only if $\varphi$ has separable reduction at $x$. Applying this observation to \eqref{Eq: tau inequalities} shows that these conditions are equivalent to saying $\tau_\varphi(x) = 0$. This finishes the proof of~\eqref{Item: Vanishing} when $x$ is a type~II point in the ramification locus. 
	
	Now let $x \in \Ram_\varphi$ be an arbitrary point. Let $K / k$ be an extension of algebraically closed and complete non-Archimedean fields such that $x_K = \iota_k^K(x)$ is a type~II point of $\Berk_K$. (See  \cite[\S4]{Faber_Berk_RamI_2011} for notation.) Write $\tau_K: \HH_K \to \RR$ for the visible ramification of the function $\varphi_K \in K(z)$ given by extension of scalars. The coefficients of the polynomials $A_\ell$ lie in $k$, and hence $\tau_K(x_K) = \tau_\varphi(x)$. Our work in the previous paragraph implies that $\tau_K(x_K) = 0$ if and only if $\varphi_K$ has separable reduction at $x_K$, which by definition is equivalent to saying that $\varphi$ has separable reduction at $x$. Thus part~\eqref{Item: Vanishing} of the proposition is proved. 
	
	Last, we prove~\eqref{Item: Support}. Since $\Berk \smallsetminus \Ram_\varphi$ is open, and since $\tau_\varphi$ vanishes identically there by definition, we must have $\supp(\tau_\varphi) \subset \Ram_\varphi$. 
	
	If $x$ is a strong interior point of the ramification locus $\Ram_\varphi$, then $\varphi$ has inseparable reduction at~$x$ \cite[Thm.~B]{Faber_Berk_RamI_2011}. It follows that $p \mid m_\varphi(x, \vec{v})$ for all directions $\vec{v} \in T_x$, and that $\tau_\varphi(x) > 0$ by part~\eqref{Item: Vanishing}.
	
	For the remainder of the proof we assume that $x \in \Ram_\varphi$ is not a strong interior point of the ramification locus. If $x \in \Ram_\varphi \smallsetminus \Hull(\Crit(\varphi))$, then $x$ is an endpoint of $\Ram_\varphi$ \cite[Cor.~7.6]{Faber_Berk_RamI_2011}. So $x$ must be of type~II or type~IV, and there is a segment $(x,x') \subset \Ram_\varphi$ consisting entirely of points at which $\varphi$ has inseparable reduction \cite[Prop.~7.3]{Faber_Berk_RamI_2011}. Then $\tau_\varphi$ is positive along this whole segment, and hence $x \in \supp(\tau_\varphi)$. As $p \mid m_\varphi(y)$ for each $y \in (x,x')$, we see that $p \mid m_\varphi(x, \vec{v})$, where $\vec{v}$ is the direction toward $x'$ \cite[Prop.~3.9]{Faber_Berk_RamI_2011}. 
	
	Finally, suppose that $x \in \Ram_\varphi \cap \Hull(\Crit(\varphi))$. After passing to an extension field $K / k$ if necessary and replacing $x$ by $x_K$ as above, we may assume that $x$ is a type~II point. If $p \mid m_\varphi(x, \vec{v})$ for some direction $\vec{v}$, then there is a point $x' \in \BB_x(\vec{v})^-$ such that $\varphi$ has inseparable reduction at all points of the segment $(x,x')$ \cite[Lem.~7.2]{Faber_Berk_RamI_2011}. Hence $x \in \supp(\tau_\varphi)$. But if there does not exist $\vec{v} \in T_x$ such that $p \mid m_\varphi(x, \vec{v})$, then in particular $\varphi$ has separable reduction at $x$. There are only finitely many directions $\vec{v} \in T_x$ such that $\Ram_\varphi \cap \BB_x(\vec{v})^-$ is nonempty since each such direction must contain a critical point \cite[Thm.~A]{Faber_Berk_RamI_2011}. Let $\vec{v}_1, \ldots, \vec{v}_n$ be a complete list of these directions. Choose points $x_1, \ldots, x_n$ near $x$ so that $x_i \in \Hull(\Crit(\varphi)) \cap \BB_x(\vec{v}_i)^-$, so that $p \nmid m_\varphi(y)$ for any $y \in (x, x_i)$, $i = 1, \ldots, n$, and so that the segment $(x, x_i)$ contains no branch point of the connected hull of the critical points. Set $\varepsilon = \min \{\rho(x, x_i) : i = 1, \ldots, n\}$, and write $V_\varepsilon = \{y \in \HH : \rho(x,y) < \varepsilon\}$. Our construction shows $V_\varepsilon\cap \Hull(\Crit(\varphi))$ contains no point with inseparable reduction, and hence $V_\varepsilon \cap \Ram_\varphi = [x, x_1) \cup \cdots \cup [x, x_n)$ by \cite[Prop.~6.6]{Faber_Berk_RamI_2011}. In particular, $V_\varepsilon$ contains no strong interior point of $\Ram_\varphi$, so that $\tau_\varphi$ is identically zero on $V_\varepsilon$. Hence $x \not\in \supp(\tau_\varphi)$.
\end{proof}

\begin{proof}[Proof of Proposition~\ref{Prop: Ramification Bounded by Tau}]
	Observe that $U \cap \Hull(\Crit(\varphi)) = \emptyset$ because $\vec{v} \not\in E_{\varphi,x}$. If $U \cap \Ram_\varphi$ is nonempty, then $\varphi$ has inseparable reduction at $x$ \cite[Prop.~6.6]{Faber_Berk_RamI_2011}. Conversely, if $\varphi$ has inseparable reduction at $x$, then $x$ is an interior point of $\Ram_\varphi$ for the strong topology \cite[Prop.~7.4]{Faber_Berk_RamI_2011}; it follows that $U \cap \Ram_\varphi$ is nonempty. Hence $\varphi$ has separable reduction at $x$ if and only if $U \cap \Ram_\varphi = \emptyset$ if and only if $\tau_\varphi(x) = 0$ (Proposition~\ref{Prop: Tau properties}). This completes the proof when $\varphi$ has separable reduction at $x$. 
	
	Now suppose that $\varphi$ has inseparable reduction at $x$, so that $U \cap \Ram_\varphi$ is nonempty. Let $\zeta_{a,r} \in U \cap \Ram_\varphi$ be a type~II point. As each connected component of $\Ram_\varphi$ contains a critical point \cite[Thm.~A]{Faber_Berk_RamI_2011}, we have that $[\zeta_{a,r}, x] \subset \Ram_\varphi$. Choose $\varepsilon >0$ sufficiently small so that $1+\varepsilon \in |k^\times|$ and so that $\zeta_{a, r(1+ \varepsilon)}$ lies on the segment $ (\zeta_{a,r}, x)$. Write $\xi = \zeta_{a, r(1+ \varepsilon)}$. As in the proof of Lemma~\ref{Lem: Diameter bound}, we may select a tangent direction $\vec{w} \in T_\xi$ such that $\varphi(\BB_\xi(\vec{w})^-) \neq \Berk$. Let $y \in \BB_\xi(\vec{w})^-$ be any classical point. Then $\varphi(\DD(y, r_y)) \neq \Berk$ by construction, and so Lemmas~\ref{Lem: Ramification Segment} and~\ref{Lem: Small Root Bound} imply that
	\[
		\left(\max_{\ell \geq 1} \left| \frac{A_\ell}{A_0}(x) \right|^{1 / \ell} \right)^{-1} \leq r_y' = r_y 
			< r(1+ \varepsilon). 
	\] 
Dividing by $\diam(x)$ and applying $-\log_{q_k}$ gives
	\[
		\rho(x, \xi) = \log_{q_k} \frac{\diam(x)}{r(1+\varepsilon)} < \tau_\varphi(x).
	\]
This constructs works for all sufficiently small $\varepsilon > 0$; hence, 
	\[
		\rho(x, \zeta_{a,r}) = \log_{q_k} \frac{\diam(x)}{r} \leq \tau_\varphi(x).
	\]
We conclude that $\rho(x, \zeta) \leq \tau_\varphi(x)$ for all $\zeta \in U \cap \Ram_\varphi$ since type~II points are dense in this set for the strong topology.

	It remains to show that when $\vec{v} \not\in E_{\varphi,x}'$, there exists a point $\zeta \in U\cap \Ram_\varphi$ such that $\rho(x,\zeta) = \tau_\varphi(x)$. In fact, Lemmas~\ref{Lem: Ramification Segment} and~\ref{Lem: Small Root Bound} imply that the point $\zeta_{y, r_y}$ from the last paragraph does the trick. 
\end{proof}


\subsection{Alternate Definition of the Visible Ramification}
\label{Sec: Alternate}

	To calculate the value of the visible ramification at a point $x \in \HH$ using the definition at the beginning of the last section, one is required to know \textit{a priori} if $x$ lies in the ramification locus or not. The ramification locus can be quite complicated, so it is desirable to avoid this issue when computing examples. To that end, define a function $\mathfrak{t}_\varphi : \HH \to \RR$ by the formula
	\[
		\mathfrak{t}_\varphi(x) = \log_{q_k} \diam(x) 
			+ \max_{\ell \geq 1} \log_{q_k}\left| \frac{A_\ell}{A_0}(x) \right|^{1/\ell}.
	\]

\begin{prop}
\label{Prop: Alternate}
	Let $\varphi \in k(z)$ be a nonconstant rational function. Then
	\[
		\mathfrak{t}_\varphi(x)  \begin{cases}
					\geq 0 & \text{if } x \in \Ram_\varphi \\
					= 0 & \text{if } x \in \Hull(\Crit(\varphi)) \smallsetminus \Ram_\varphi \\
					< 0 & \text{if } x \not\in \Hull(\Crit(\varphi)) \cup \Ram_\varphi.
			\end{cases}
	\] 
In particular, we have the identity of functions $\tau_\varphi = \max\{\mathfrak{t}_\varphi, 0\}$.
\end{prop}

\begin{remark}
	In practice, for a given $\varphi$ it is straightforward to compute $\Ram_\varphi \cap \Hull(\Crit(\varphi))$ using the Algebraic Reduction Formula \cite[Prop.~3.14]{Faber_Berk_RamI_2011} and explicit coordinate changes. It is a combinatorial exercise to compute $\mathfrak{t}_\varphi$. The above proposition shows $\Ram_\varphi \smallsetminus \Hull(\Crit(\varphi))$ can then be determined by looking at the locus where $\mathfrak{t}_\varphi$ is nonnegative. Compare with Example~\ref{Ex: Pathology}. 
\end{remark}

\begin{proof}
	As in the proof of Proposition~\ref{Prop: Tau properties}, we may replace $k$ with a larger field if necessary in order to assume that $\HH$ consists only of type~II points. 

	For $x$ in the ramification locus, we note that $\mathfrak{t}_\varphi(x) = \tau_\varphi(x)$, so the result follows from Proposition~\ref{Prop: Tau properties}\eqref{Item: Strong Continuous}.

	For $x \in \Hull(\Crit(\varphi)) \smallsetminus \Ram_\varphi$, choose a direction $\vec{v} \in T_x$ so that $\BB_x(\vec{v})^-$ does not intersect the ramification locus. Write $x = \zeta_{y,R}$ for some $y \in k$ and $R > 0$, where $y$ is not a pole of $\varphi$.  Now $\varphi$ is injective on the open disk $\DD(y,R)^-$ since it does not meet $\Ram_\varphi$ \cite[Cor.~3.8]{Faber_Berk_RamI_2011}. As $x$ lies on the connected hull of the critical points, there exists another direction $\vec{w}$ such that $\BB_x(\vec{w})^-$ does not contain $\infty$, but it does contain a critical point. In particular, the surplus multiplicity $s_\varphi(\BB_x(\vec{w})^-)$ is positive \cite[Prop.~3.18]{Faber_Berk_RamI_2011}, which means $\varphi(\BB_x(\vec{w})^-) = \Berk$. So there exists a type~I point $y_1 \in \BB_x(\vec{w})^-$ such that $\varphi(y_1) = \varphi(y)$. The final two paragraphs of the proof of Lemma~\ref{Lem: Diameter bound}  now apply \textit{mutatis mutandis} to show that $r_y' = \diam(x) = R$. By Lemma~\ref{Lem: Small Root Bound}, this means
	\[
		\mathfrak{t}_\varphi(x) = \log_{q_k} \frac{\diam(x)}{r_y'} = 0. 
	\]
	
	For $x \not\in \Hull(\Crit(\varphi)) \cup \Ram_\varphi$, choose a direction $\vec{v} \in T_x$ such that $\BB_x(\vec{v})^- \cap \Ram_\varphi = \emptyset$ and write $x = \zeta_{y, R}$ for some $y \in \BB_x(\vec{v})^-$ and $R > 0$. Note that $\vec{v} \not\in E_{\varphi,x}'$. As the ramification locus is closed, there exists $\varepsilon > 0$ such that $\DD(y, R + \varepsilon)^-$ does not intersect $\Ram_\varphi$. So $\varphi$ is injective on $\DD(y, R + \varepsilon)^-$ \cite[Prop.~3.18]{Faber_Berk_RamI_2011}. In the notation of the previous subsections, $A_\varphi(z,y)$ has no root in the disk $D(0, R + \varepsilon)^-$, and so $r_y' \geq R + \varepsilon$. Equivalently, by Lemma~\ref{Lem: Small Root Bound} we see that
	\[
		\mathfrak{t}_\varphi(x) = \log_{q_k} \frac{R}{r_y'} < \log_{q_k} \frac{R + \varepsilon}{r_y'} \leq 0.
	\]
Hence $\mathfrak{t}_\varphi(x) < 0$, as desired.
\end{proof}



\section{Proof of Theorem~D and the Applications}
\label{Sec: Proofs}

	Theorem~D depends heavily on a non-Archimedean version of Rolle's Theorem that is valid for series converging on a suitably large disk. For the reader's convenience, we give its proof, which is due to Robert in the case  $k = \CC_p$.
		
\begin{prop}[{\cite[\S2.4]{Robert_p-adic_Book_2000}}]
\label{Prop: Second Mapping Lemma}
	Suppose $k$ has characteristic zero and residue characteristic $p > 0$. 
	Let $\varphi \in k(z)$ be a nonconstant rational function satisfying the following three hypotheses: 
		\begin{itemize}
			\item $0 < p \leq \deg(\varphi)$;
			\item $\varphi$ has no pole in the classical disk $D(0,1)^-$; and
			\item $\varphi$ has no critical point in the classical disk $D(0,1)^-$.
		\end{itemize}
Then $\varphi$ is injective on the disk $D(0,r_p)^-$, where $r_p = |p|^{1/(p-1)} < 1$.
\end{prop}


\begin{proof}
As $\varphi$ has no pole in $D(0,1)^-$, it can be expanded as a series $\varphi(z) = a_0 + a_1 z + a_2 z^2  + \cdots$ that converges  on $D(0,1)^-$. After a change of coordinate on the target, we may further assume that $\varphi(0) = a_0 = 0$. 

    Note $a_1 \neq 0$, else the origin is a critical point. If $a_m = 0$ for all $m \geq 2$, then $\varphi(z) = a_1z$ is evidently injective. So let us suppose from now on that there exists $m \geq 2$ with $a_m \neq 0$. Since $\varphi'(z) = a_1 + 2a_2 z + 3a_3z^2 + \cdots$ has no zero in $D(0, 1)^-$, the first segment of its Newton polygon must have nonnegative slope. The Newton polygon is a lower convex hull, so the slope of the segment from $(0, \ord_k(a_1))$ to $(m-1, \ord_k(ma_m))$ will be bounded below by the slope of the first face, which is to say that
	\[
		\frac{\ord_k(ma_m) - \ord_k(a_1)}{m - 1} = \frac{\ord_k(a_m) - \ord_k(a_1)	}{m-1} + \frac{\ord_p(m)}{m-1} \geq 0.
	\]
It follows that
	\be
	\label{Eq: Segment slope}
		\frac{\ord_k(a_m) - \ord_k(a_1)	}{m-1}  \geq  \frac{-\ord_p(m)}{m-1}.
	\ee
	
	Write $\mu$ for the slope of the segment from $(1, \ord_k(a_1))$ to $(m, \ord_k(a_m))$. If $p \nmid m$, then \eqref{Eq: Segment slope} shows that $\mu \geq 0$. Otherwise, write $m = ep^s$ for some positive integers $e, s$ with $e$ coprime to $p$. Then
	\benn
		\ba
		\frac{\ord_k(a_m) - \ord_k(a_1)	}{m-1} \geq \frac{-s}{ep^s - 1} \geq \frac{-s}{p^s - 1} &= \frac{-s}{(p-1)(p^{s-1} + \cdots + p + 1)} \\
		&\geq \frac{-s}{(p-1)(1 + \cdots + 1)} = \frac{-1}{p-1}.
		\ea
	\eenn
If we choose $m$ so that $\mu$ is the slope of the first non-vertical face of the Newton polygon of $\varphi$, then we have proved that $\mu \geq -1 / (p-1)$. Therefore $\varphi$ has exactly one root in the disk $D(0, r_p)^-$. Since an analytic map is everywhere $n$-to-1 on a disk on which it has no pole (for some $n \geq 1$), we see that $\varphi$ is injective on the disk $D(0, r_p)^-$.
\end{proof}

\begin{proof}[Proof of Theorem~D]
	By Proposition~\ref{Prop: Ramification Bounded by Tau} and the fact that $\tau_\varphi$ is continuous, it suffices to prove that 
	\[
		\tau_\varphi(x) \begin{cases} 
			= 0 & \text{ if $p = 0$ or  $p > \deg(\varphi)$} \\
			\leq \frac{1}{p-1} & \text{ if $0 < p \leq \deg(\varphi)$}.
		\end{cases}
	\]

	First suppose that $p = 0$ or $p > \deg(\varphi)$. Then $m_\varphi(x) \leq \deg(\varphi)$ cannot be divisible by $p$, and so $\varphi$ must have separable reduction at $x$. Thus $\tau_\varphi(x) = 0$ (Proposition~\ref{Prop: Tau properties}). 
	
	Now suppose that $0 < p \leq \deg(\varphi)$. Let $x \in \Hull(\Crit(\varphi))$ be a type~II point, and let $\vec{v} \in T_x \smallsetminus E_{\varphi,x}'$ be a tangent direction such that $\varphi(\BB_x(\vec{v})^-) \neq \Berk$. Evidently such a direction must exist since otherwise $\varphi$ would have infinitely many poles. If $\tau_\varphi(x) = 0$, then we are finished. Otherwise, $\tau_\varphi(x) > 0$ and so $U \cap \Ram_\varphi$ is nonempty, where we write $U = \BB_x(\vec{v})^-$. If $\tau_\varphi(x) > 1 / (p-1)$, then there exists $\zeta \in U \cap \Ram_\varphi$ such that $\rho(\zeta, x) > 1 / (p-1)$. Since $U \cap \Ram_\varphi$ is connected \cite[Prop.~6.6]{Faber_Berk_RamI_2011}, we may assume without loss that $\zeta$ is of type~II. After a change of coordinates on the source and target, we may assume that $U = \DD(0,1)^- = \varphi(U)$ and that $\zeta = \zeta_{0, r}$ for some $0 < r < r_p$, where $r_p = |p|^{1 / (p-1)}$. By construction, $D(0,1)^-$ contains no critical point of $\varphi$ and no pole of $\varphi$, and so Proposition~\ref{Prop: Second Mapping Lemma} implies that $\varphi$ is injective on the disk $D(0, r_p)^-$. Since $\zeta \in \DD(0, r_p)^-$, we have obtained a contradiction. 
\end{proof}

	The two applications stated in the introduction have identical conclusions. This allows us to prove the contrapositive of both simultaneously: if $\varphi$ has no critical point in the closed disk $D(a, r \cdot \gamma_{p,d})$, then $\varphi$ is injective on the disk $D(a,r)$ and $\varphi\left(D(a,r)\right) \neq \PP^1(k)$. (Recall that  $\gamma_{p,d} = 1$ if $p = 0$ or $p > d$ and $\gamma_{p,d} = |p|^{-1/ (p-1)} > 1$ otherwise.)

\begin{proof}[Proof of the Applications]
	We begin by reducing to the case $D(a,r) = D(0,1)$. Choose $t \in k$ such that  $|t| \geq r$, and consider the rational function $\psi(z) = \varphi(tz + a)$. Then $\psi$ satisfies the hypotheses of the applications relative to the disk $D(0,1)$. If the applications are true for $\psi$, then there exists a critical point of $\psi$ in the disk $D(0, \gamma_{p,d})$. Equivalently, there exists a critical point of $\varphi$ in the disk $D(a, |t| \cdot \gamma_{p,d})$. Now let $|t| \to r$; as $\varphi$ has only finitely many critical points, the conclusion must continue to hold for $D(a, r \cdot \gamma_{p,d})$. 
    
	Now we prove the contrapositive when $D(a,r) = D(0,1)$. Suppose that $\varphi$ has no critical point in the closed disk $D(0, \gamma_{p,d})$. Then the closed Berkovich disk $\DD(0, \gamma_{p,d})$ is disjoint from the connected hull of the critical points. Let $U$ be the connected component of $\Berk \smallsetminus \Hull(\Crit(\varphi))$ containing $\DD(0, \gamma_{p,d})$, and write $U = \BB_x(\vec{v})^-$ for some type~II point $x \in \Hull(\Crit(\varphi))$ and $\vec{v} \in T_x \smallsetminus E_{\varphi,x}$. 	Then
	\benn
		\ba
			\rho(x, \zeta_{0,1}) &= \rho(x, \zeta_{0, \gamma_{p,d}}) + \rho(\zeta_{0, \gamma_{p,d}}, \zeta_{0,1}) \\
				&>   \log_{q_k} \gamma_{p,d} \\
				 &= \begin{cases}
				 		0 & \text{ if $p = 0$ or $p > d$} \\
						\frac{1}{p-1} & \text{ if $0 < p \leq d$}.
				 	\end{cases}
		\ea
	\eenn
By Theorem~D, $\zeta_{0,1}$  does not lie in the ramification locus. It follows that the closed disk $\DD(0,1)$ is disjoint from the ramification locus \cite[Thm.~A]{Faber_Berk_RamI_2011}. 

	Now we observe that, since $\Ram_\varphi$ is closed, there exists $\varepsilon > 0$ such that $\DD(0, 1 + \varepsilon)^-$ is disjoint from the ramification locus. Hence $\varphi$ is injective on $\DD(0, 1 + \varepsilon)^-$ \cite[Cor.~3.8]{Faber_Berk_RamI_2011}. It follows that $\varphi$ is injective on $\DD(0,1)$, and so also on the classical disk $D(0,1)$. In particular, $\varphi$ can have at most one zero in this disk, which completes the proof of Application~\ref{App: Rolle's Theorem}. 
	 
	Since $\DD(0, 1+\varepsilon)^-$ is disjoint from the ramification locus, we see that $m_{\varphi}(\zeta_{0, 1 + \varepsilon}, \vec{0}) = 1$, where $\vec{0}$ denotes the tangent direction pointing toward $0$. Then $\varphi(\DD(0, 1 + \varepsilon)^-) = \Berk$ if and only if the surplus multiplicity $s_\varphi(\DD(0,  1 + \varepsilon)^-)$ is positive, which is equivalent  to the existence of a critical point in $\DD(0, 1 + \varepsilon)^-$ \cite[Prop.~3.18]{Faber_Berk_RamI_2011}. But $\DD(0,1+\varepsilon)^-$ is disjoint from the ramification locus, and therefore $\varphi(\DD(0,1)) \neq \Berk$. Restricting to classical points, we see that $\varphi(D(0,1)) \neq \PP^1(k)$, and the proof of Application~\ref{App: Exploding Disk} is finished. 
\end{proof}


\section{Behavior Near a Critical Point}
\label{Sec: Local Fuzz}
	
	The shape of the ramification locus in a weak neighborhood of a critical point depends only on its multiplicity and on the characteristic of the residue field of $k$. This is the essence of Theorem~F from the introduction, whose proof occupies most of this section. The proof of Theorem~E appears at the end.

	Suppose first that $\varphi$ has only one critical point. Then the Hurwitz formula shows $k$ must have characteristic~$p$ and $p \mid m_\varphi(c)$ \cite[Cor.~2.5]{Faber_Berk_RamI_2011}. On one hand, this means $\Hull(\Crit(\varphi)) = \{c\}$, and hence $\Hull(\Crit(\varphi)) + r = \{c\}$ for each $r > 0$ since $c$ lies at infinite distance from $\Berk \smallsetminus \PP^1(k)$. On the other hand, we know that the ramification locus has no isolated point for the weak topology \cite[Cor.~3.13]{Faber_Berk_RamI_2011}. So $\Ram_\varphi \neq \{c\}$, and the proof is complete in this special case.

	For the remainder of this section we will assume that $\varphi$ has at least two distinct critical points. After a suitable change of coordinate on the source and target, we may assume
	\begin{itemize}
		\item $c = 0 = \varphi(0)$;
		\item $\varphi$ has no pole in $D(0,1)$;
		\item $\varphi$ has no critical point in $D(0,1)$ other than $c = 0$; and 
		\item $m_\varphi(\zeta_{0,r}) = m_\varphi(0)$ for all $r \in [0,1]$.
	\end{itemize}
The penultimate requirement is possible because $\varphi$ has only finitely many critical points, which in turn is true because $\varphi$ is separable. The final assertion is possible by \cite[Prop.~3.9, 3.12]{Faber_Berk_RamI_2011}. Let $m = m_\varphi(0) > 1$, and let $m' \geq m$ be chosen so that $m' - 1$ is the weight of the critical point~0. (Recall that the weight of $0$ is defined to be $\ord_{z=0}(\varphi')$.)

	After replacing $\varphi$ by a suitable scalar multiple, it can be represented on the disk $D(0,1)$ by the series:
		\[
			f(z) = z^m \ ( 1  + \varepsilon(z)), \qquad  \varepsilon(z) = \varepsilon_1z + \varepsilon_2 z^2 + \cdots,
		\]
where $|\varepsilon_i| < 1$ for $i = 1, 2, \ldots$. 
Therefore
	\benn
		f'(z) = z^{m-1} \ \left(m + m \varepsilon(z) + z\varepsilon'(z)\right)
	\eenn
	If $\mathrm{char}(k) = 0$ or if $\mathrm{char}(k) = p > 0$ and $p \nmid m$, then $\ord_{z=0} f'(z) = m - 1$. If $\mathrm{char}(k) = p > 0$ and $p \mid m$, then the definition of $m'$ shows
	\[
		f'(z) = z^m \varepsilon'(z) = z^{m' - 1} \ \left( m' \varepsilon_{m' - m} + (m'+1) \varepsilon_{m'-m+1}z + \cdots \right),
	\]
where $\varepsilon_{m' - m} \neq 0$ and $p \nmid m'$. (Note $\varepsilon'(z) \not\equiv 0$ by separability.) Moreover, since $\varphi$ has no other critical point in $D(0,1)$, it follows that $f'(z) / z^{m' - 1}$ is invertible on this disk; that is, $|\ell \varepsilon_\ell| < | m' \varepsilon_{m' - m}|$ for all $\ell > m' - m$.

	Our normalizations thus far show that
	 \[
	 	[0, \zeta_{0,1}) = \Hull(\Crit(\varphi)) \cap \DD(0,1)^- \subset \Ram_\varphi \cap \DD(0,1)^-,
	\]
because $\varphi$ must have at least one critical point outside the disk $\DD(0,1)^-$. It remains to determine the remaining ramified points in $\DD(0,1)^-$.

	Here is the general strategy for the remainder of the proof. Choose $\delta \in k^{\circ \circ}$. The function $F(z) = f(z + \delta) - f(\delta)$ is a convergent series on $D(0,1)$, and the number of roots of $F$ is $D(0,r)$ agrees with the multiplicity $m_\varphi(\zeta_{\delta, r})$ for $r \leq 1$ (Lemma~\ref{Lem: Multiplicity and zeros}). We will show that the number of roots is~1 unless $r \geq |\delta|$ (resp. $r \geq |p|^{1 / (p-1)} |\delta|$) if $p \nmid m$ (resp. $\mathrm{char}(k) = 0$ and $ p \mid m$). We can use the Newton polygon of $F$ to show the number of roots is as claimed, but this will require a fair amount of calculation. A similar strategy will be used when $\mathrm{char}(k) = p$ and $p \mid m$.


	We first compute $F(z) = f(z+ \delta) - f(\delta)$ for arbitrary $p, m$:
	\benn
		\ba
			F(z) &= (z+\delta)^m - \delta^m 
					+ (z + \delta)^m \ \sum_{i \geq 1} \varepsilon_i (z+\delta)^i 
					- \delta^m \sum_{i \geq 1} \varepsilon_i \delta^i \\
			&= (z + \delta)^m - \delta^m + \sum_{\ell = 0}^m \binom{m}{\ell} \delta^{m- \ell} z^\ell \ 
					\sum_{i \geq 1} \varepsilon_i \sum_{j = 0}^i \binom{i}{j} \delta^{i-j} z^j
					- \delta^m \sum_{i \geq 1} \varepsilon_i \delta^i \\
			&= (z+ \delta)^m - \delta^m + \sum_{t \geq 0} \sum_{\ell = 0}^{\min\{t,m\}} 
					\sum_{i \geq 1} \binom{m}{\ell} \binom{i}{t - \ell} 
					\delta^{m - t + i} \varepsilon_i z^t - \delta^m \sum_{i \geq 1} \varepsilon_i \delta^i,
		\ea
	\eenn
where we have set $t = j + \ell$ and used the fact that $\binom{i}{t-\ell} = 0$ whenever $t- \ell > i$. Now observe that the $t = 0$ term in the sum cancels with the final summation. To simplify this expression further, let us write  $O(z^{m+1})$ to denote an infinite series with order of vanishing at least $m+1$ and all coefficients in $k^{\circ \circ}$. Then we have
	\benn
		\ba
			F(z) &= (z+ \delta)^m - \delta^m + \sum_{t = 1}^m \sum_{\ell = 0}^t 
					\sum_{i \geq 1} \binom{m}{\ell} \binom{i}{t - \ell} 
					\delta^{m - t + i} \varepsilon_i z^t + O(z^{m+1}) \\
			&= (z+ \delta)^m - \delta^m + \sum_{t =1}^m \sum_{i \geq 1} \binom{m}{t} \delta^{m - t+i} \varepsilon_iz^t \\
			& \qquad + \sum_{t=1}^m \sum_{\ell = 0}^{t-1} \sum_{i \geq 1} 
					\binom{m}{\ell} \binom{i}{t- \ell} \delta^{m-t+i} \varepsilon_i z^t + O(z^{m+1}) \\
			&= \underbrace{(z+\delta)^m - \delta^m}_{(A)} + \underbrace{\varepsilon(\delta) \ \sum_{t=1}^m \binom{m}{t} \delta^{m-t} \  z^t}_{(B)} \\
			& \qquad + \underbrace{\sum_{t=1}^m \delta^{m-t+1} \left[ \sum_{\ell = 0}^{t-1} \sum_{i \geq 1} 
					\binom{m}{\ell} \binom{i}{t- \ell} \delta^{i-1} \varepsilon_i \right]z^t}_{(C)} + O(z^{m+1}).
		\ea
	\eenn	

Write $F(z) = \sum f_t z^t$.  Evidently $f_0 = 0$. For the coefficient $f_t$ with $1 \leq t \leq m$, we see that
	\be
	\label{Eq: Contributions}
		\ba
			\text{(A)} & \text{ contributes } \binom{m}{t}\delta^{m-t}  \\
			\text{(B)} & \text{ contributes }  \binom{m}{t}\delta^{m-t} \varepsilon(\delta) \\
			\text{(C)} & \text{ contributes }  \delta^{m-t+1} \cdot (\text{element in $k^{\circ \circ}$}).
		\ea
	\ee	
We also see that 
	\[
		f_m = 1 + \varepsilon(\delta) + \delta \cdot (\text{element in $k^{\circ \circ}$}) \Rightarrow |f_m| = 1.
	\]
Note that $(m, 0)$ must be a vertex of the Newton polygon of $F$ because all of the coefficients $f_t$ with $t \neq m$ lie in $k^{\circ \circ}$.
Our treatment of the remaining coefficients will depend on which case of the proposition we are trying to prove. 

\subsection{Case $\mathrm{char}(k) = 0$}

	For $1 \leq t < m$, we see that the contribution of (B) is strictly dominated by that from (A). Define $r_0 = \min\left\{ \left|\binom{m}{\ell}\right| \ : \ 0 \leq \ell \leq m \right\}$. If we require that $|\delta| < r_0$, then the contribution from (C) is bounded above by $|\delta| \cdot |\delta^{m-t}| < \left|\binom{m}{t} \delta^{m-t} \right|$. (This restriction has the effect of shrinking the disk $U$ in the theorem.) Hence 
	\[
		|f_t| = \left|\binom{m}{t} \delta^{m-t} \right|, \qquad 1 \leq t \leq m.
	\]

	We now compute the slope of the first non-vertical face of the Newton polygon of $F$. It is given by
	\benn
		\ba
			\min_{2 \leq t \leq m} \frac{\ord_k \left(\binom{m}{t} \delta^{m-t} \right) -  \ord_k \left(\binom{m}{1} \delta^{m-1} \right)}{t - 1} 	
				&= -\ord_k(\delta) + \min_{2 \leq t \leq m} \frac{\ord_p \left( \binom{m}{t} / m \right)}{t-1} \\
				&= -\ord_k(\delta) - \begin{cases}  0 & \text{if $p \nmid m$} \\ \frac{1}{p-1} & \text{if $p \mid m$},
					\end{cases}
		\ea
	\eenn
where the final equality is a consequence of the following lemma, whose proof will be given at the end of this subsection.

\begin{lem}
\label{Lem: Binomial Valuations}
	Let $p$ be a prime and $m \geq 2$ a natural number. Then
		\[
			\min_{2 \leq s \leq m} \frac{\ord_p\left( \binom{m}{s} / m\right)}{s-1} =
				\begin{cases}
					0 & \text{if $p \nmid m$} \\
				 - \frac{1}{p-1} & \text{if $p \mid m$}.
				\end{cases}
		\]
	If $p \nmid m$ (resp. $p \mid m$), then the minimum is attained at $s = m$ (resp. at $s =p$).
\end{lem}

    If $p \nmid m$, we have shown that $F$ has a single root in the open disk $D(0, |\delta|)^-$ and $m$ roots in the closed disk $D(0, |\delta|)$. A change of coordinate and an application of Lemma~\ref{Lem: Multiplicity and zeros} shows that
    \[
        m_\varphi(\zeta_{\delta, r}) = \begin{cases} 1 & r < |\delta| \\ m & r = |\delta| \end{cases}.
    \]
Since $\zeta_{0, |\delta|} = \zeta_{\delta, |\delta|}$, we conclude that $\Ram_\varphi \cap \DD(0, r_0)^- = [0, \zeta_{0, r_0})$. 

If $p \mid m$, then $F$ has a single root in the open disk $D(0, |p|^{1/(p-1)} |\delta|)^-$ and $p$ roots in the closed disk $D(0, |p|^{1/(p-1)} |\delta|)$. Hence
    \[
         m_\varphi(\zeta_{\delta, r})  \begin{cases} = 1 & r < |p|^{1/(p-1)} |\delta| \\ \geq p  & |p|^{1/(p-1)} \leq r \leq |\delta|  \end{cases},
    \]        
so as above we see that  
    \[
        \Ram_\varphi \cap \DD(0, r_0)^- = [0, \zeta_{0,r_0}) + \frac{1}{p-1}.
    \]

\begin{proof}[Proof of Lemma~\ref{Lem: Binomial Valuations}]
	If $p \nmid m$, then
		\[
			\ord_p\left( \binom{m}{s} / m\right)  = \ord_p\left( \binom{m}{s}\right) \geq 0
		\]
by integrality of binomial coefficients. Equality holds if $s = m$.

	Now suppose that $p \mid m$. Then
		\benn
			\ba
				\ord_p\left(  \binom{m}{s} / m \right)  &=
					\ord_p \left( \frac{(m-1) \cdots (m - (s-1)) }{s!} \right) \\
					&\geq - \ord_p(s!) \\
					&= - \frac{s - M_s}{p-1}, 
			\ea
		\eenn
where $M_s$ is the sum of the digits of $s$ written in base $p$. (This last formula is quickly established by induction on $s$ \cite[\S5.3.1]{Robert_p-adic_Book_2000}.) Thus
	\benn
		\ba
			\frac{\ord_p\left( \binom{m}{s} / m \right)}{s-1} 	
				&\geq - \frac{1}{p-1} \cdot \frac{s-M_s}{s-1} \\
				&\geq - \frac{1}{p-1}.
		\ea
	\eenn
Setting $s = p$ gives equalities throughout. 
\end{proof}

\subsection{Case $\mathrm{char}(k) = p > 0$ and $p \nmid m$}

	First we look at the linear coefficient $f_1$, for which we see that
	\benn
		\ba
			\text{(A)} & \text{ contributes } m\delta^{m-1}  \\
			\text{(B)} & \text{ contributes }  m\delta^{m-1} \varepsilon(\delta) \\
			\text{(C)} & \text{ contributes }  \delta^m \cdot (\text{element in $k^{\circ \circ}$}).
		\ea
	\eenn
As $|\delta| < 1$, evidently (A) is the dominant contribution, so we have $|f_1| = |\delta^{m-1}|$.

	Now we turn to the coefficient $f_t$ with $1 < t < m$. If $p \mid \binom{m}{t}$, then the contributions from \eqref{Eq: Contributions}(A) and (B) vanish and so $|f_t| < |\delta|^{m-t + 1}$. 
If $p \nmid \binom{m}{t}$, then the contribution from (A) strictly dominates the other two. Hence $|f_t| = |\delta|^{m-t}$ when $p \nmid \binom{m}{t}$. In either case, we have $|f_t| \leq |\delta|^{m-t}$.

	We claim that the first non-vertical face of the Newton polygon of $F$ is the one connecting $(1, \ord_k(f_1))$ to $(m, \ord_k(f_m))$. The slope of the segment from $(1, \ord_k(f_1))$ to $(t, \ord_k(f_t))$ is
	\benn
		\frac{\ord_k(f_t) - \ord_k(f_1)}{t-1} \geq \frac{(m-t) \ord_k(\delta) + (m-1)\ord_k(\delta)}{t-1}
			= -\ord_k(\delta),
	\eenn
with equality if $t = m$. By the same argument we used in the previous case, we see that $[0, \zeta_{0,1}) = \Ram_\varphi \cap \DD(0,1)^-$. 



\subsection{Case $\mathrm{char}(k) = p > 0$ and $p \mid m$}

	Since $m = 0$ in $k$, the linear coefficient of $F(z)$ comes from \eqref{Eq: Contributions}(C):
	\[
		f_1 = \delta^m \sum_{i \geq 1} i \delta^{i-1} \varepsilon_i  = \delta^m \varepsilon'(\delta).
	\]
Since $|\varepsilon'(\delta)| = |\delta^{m' -m- 1} \varepsilon_{m' - m}|$, we see that
	\[
		|f_1| = |\delta^{m'  - 1} \varepsilon_{m' - m}|.
	\]
	
	An upper bound for the slope of the first non-vertical face of the Newton polygon of $F$ is given by the slope of the segment connecting $(1, \ord_k(f_1))$ to $(m, \ord_k(f_m))$:
	\benn
		\ba
		\frac{\ord_k(f_m) - \ord_k(f_1)}{m-1} &= \frac{-m \ord_k(\delta) - \ord_k(\varepsilon'(\delta))}{m-1} \\
			&= - \frac{m'-1}{m-1} \ord_k(\delta) - \frac{\ord_k(\varepsilon_{m' - m})}{m-1} \\
			&<- \frac{m'-1}{m-1} \ord_k(\delta).
		\ea
	\eenn 
Set $r = r(\varphi, \delta) =  |\delta|^{(m' -1) / (m-1)} \cdot |\varepsilon_{m' - m}|^{1 / (m-1)}< |\delta|$. Then $F$ has only 1 root in $D(0,r)^-$ and at least~2 roots in $D(0,r)$. Thus $\zeta_{\delta, r} \in \Ram_\varphi$, while the nearest point on the connected hull of the critical points is $\zeta_{0, |\delta|}$. So we have
    \benn
        \ba
        \rho(\zeta_{\delta, r}, \zeta_{0, |\delta|}) &= \log_{q_k}\left( |\delta|  / r \right) \\
            &= \left[1 - \frac{m' - 1}{m-1} \right] \log_{q_k} |\delta|  - \frac{1}{m-1} \log_{q_k} |\varepsilon_{m' -m}|\\
            &= \frac{m' - m}{m-1} \ord_k(\delta) + \frac{1}{m-1} \ord_k(\varepsilon_{m' - m}). 
        \ea
    \eenn
This last quantity tends to infinity as $\delta \to 0$. It follows that for any $R < 1$, the set $\DD(0,R)^- \cap \Ram_\varphi$ fails to be contained in a close tube about $[0, \zeta_{0, R})$ of finite radius. In fact, Proposition~\ref{Prop: Ramification Bounded by Tau} implies that as $|\delta| \to 0$, 
    \[
        \tau_\varphi(\zeta_{0, |\delta|}) = \frac{m' - m}{m-1} \ord_k(\delta) + \frac{1}{m-1} \ord_k(\varepsilon_{m' - m})  \to +\infty.
    \]


\subsection{Proof of Theorem~E}
\label{Sec: Theorem Proofs}

    The following corollary will be used in the proof of Theorem~E. It follows immediately from Proposition~\ref{Prop: Ramification Bounded by Tau} and the proof of Theorem~F. 

\begin{cor}
\label{Cor: Tau is bounded}
	Under the hypotheses of the proposition, we find that
		\[
			\lim_{\substack{x \to c \\ x \in \Hull(\Crit(\varphi))}} \tau_\varphi(x)  = 
				\begin{cases} 
				0 & \text{if $p \nmid m$} \\
				\frac{1}{p-1} & \text{if $\mathrm{char}(k) = 0$ and $p \mid m$} \\
				+\infty & \text{if $\mathrm{char}(k) = p$ and $p \mid m$}.
			\end{cases}
		\]
\end{cor}

\begin{remark}
    The limit in the corollary is the same for both the weak and strong topologies on $\HH$. 
\end{remark}

	The proof of Theorem~E is now an easy consequence of the machinery we have developed. 
		
\begin{proof}[Proof of Theorem~E]
	Suppose first that $\varphi$ is tamely ramified at each of its critical points. Then it has at least two critical points, and without loss of generality we may assume one of them is~$\infty$. The visible ramification $\tau_\varphi$ is continuous (Proposition~\ref{Prop: Tau properties}) and constant in a neighborhood of each of the critical points of $\varphi$ (Corollary~\ref{Cor: Tau is bounded}). So $\tau_\varphi$ has a maximum value; denote it by $r_\varphi$. If $\zeta \in \Ram_\varphi \smallsetminus \Hull(\Crit(\varphi))$, then there exists a type~II point $x \in \Hull(\Crit(\varphi))$ and a direction $\vec{v} \in T_x \smallsetminus E_{\varphi,x}$ such that $\zeta \in \BB_x(\vec{v})^-$. Proposition~\ref{Prop: Ramification Bounded by Tau} shows that $\rho(\zeta, x) \leq \tau_\varphi(x) \leq r_\varphi$. Hence $\Ram_\varphi \subset \Hull(\Crit(\varphi)) + r_\varphi$, as desired.
	
	Now suppose that $k$ has positive characteristic $p$, and that there is a critical point $c$ such that $p \mid m_\varphi(c)$. Hence $\tau_\varphi(x) \to +\infty$ as $x$ approaches $c$  (Corollary~\ref{Cor: Tau is bounded}). Proposition~\ref{Prop: Ramification Bounded by Tau} shows that for each $x$ of type~II sufficiently close to $c$ there exists a point $\zeta \in \Ram_\varphi$ for which $\rho(\zeta, x) = \tau_\varphi(x)$, and hence there cannot exist a real number $r_\varphi$ for which $\Ram_\varphi \subset \Hull(\Crit(\varphi)) + r_\varphi$.
\end{proof}


\section{Examples}
\label{Sec: Examples}

	Here we give three examples. The first two show that the bounds in Theorems~D and~E cannot be made any stronger. The third illustrates some of the subtlety in the visible ramification. 
	
\begin{example}
\label{Ex: Sharp bound}
	First we show that the bound $1 / (p-1)$ in Theorem~D is the best possible in general when $0 < \deg(\varphi) \leq p$. Suppose that $k$ has characteristic zero and residue characteristic $p > 0$, and consider the polynomial function
	\[
		\varphi(z) = z^p + z.
	\]
The Newton polygon shows that $\varphi$ has a single zero in the open unit disk $D(0, 1)^-$, and it has $p$ zeros in the closed unit disk $D(0,1)$. So the Gauss point $\zeta_{0,1}$ lies in the Berkovich ramification locus (Lemma~\ref{Lem: Multiplicity and zeros}). The Newton polygon of the derivative of $\varphi$ shows that the finite critical points have absolute value $|p|^{-1 / (p-1)}$. Therefore
	\[
		\rho(\Hull(\Crit(\varphi)), \zeta_{0,1}) = \rho(\zeta_{0, |p|^{-1/(p-1)}}, \zeta_{0,1}) = \frac{1}{p-1},
	\]
which shows Theorem~D is sharp when $\deg(\varphi) = p$. A similar argument applies to the function $\varphi(z) = az^d + z^p + z$ when $d > p$ and $a \in k$ has sufficiently small absolute value (depending on $d$). 

    Note that this example also shows that the bound appearing in Application~\ref{App: Rolle's Theorem} is sharp. 
\end{example}

\begin{example}
\label{Ex: E not uniform}
	Next we give an example to show that one cannot strengthen Theorem~E: when $k$ has positive characteristic $p$,  the radius $r_\varphi$ can be arbitrarily large as one varies the function $\varphi$.  Let $\psi \in k(z)$ be any rational function of degree at least~2 that is tamely ramified at each of its critical points, and suppose further that there exists a critical point with multiplicity strictly greater than~$p$. After a change of coordinate on the source, we may assume that the smallest critical point has absolute value~1 and that $\infty$ has multiplicity larger than~$p$. A change of coordinate on the target ensures that $\psi(0) = 0$, $\psi'(0) = 1$, and $\psi(\infty) = \infty$. (For example, take $p > 2$ and $\psi(z) = z^d + z$ with $d > p$ and $d \not\equiv 0 \pmod p$.) 
	
	Let $\alpha \in k$ be an undetermined element of $k$, and define
	\[
		\varphi(z) = \psi(z) + \alpha z^p.
	\]
By construction, $\varphi$ and $\psi$ have the same critical points with the same multiplicities. In a neighborhood of the origin, we may expand $\psi(z) = z + a_2z^2 + a_3z^3 + \cdots$, so that
	\[
		\varphi(z) = z + a_2z^2 + a_3z^3 + \cdots + (\alpha + a_p)z^p + \cdots
	\]
If $|\alpha|$ is sufficiently large (depending on $\psi$), the first non-vertical face of the Newton polygon of this series has vertices $(1, 0)$ and $(p, \ord_k(\alpha))$, which shows that $\varphi$ has $p$~zeros in the closed disk of radius $r(\alpha) = |\alpha|^{-1 / (p-1)}$. Hence $\zeta_{0, r(\alpha)} \in \Ram_\varphi$ (Lemma~\ref{Lem: Multiplicity and zeros}). By construction, $\zeta_{0,1} \in \Hull(\Crit(\varphi))$, while $\varphi$ has no critical point in the disk $\DD(0,1)^-$. With the notation of Theorem~E, we find that 
	\[
		r_\varphi \geq \max_{\zeta \in \DD(0,1)^- \cap \Ram_\varphi} \rho(\Hull(\Crit(\varphi)), \zeta)
			\geq \rho(\zeta_{0,1}, \zeta_{0, r(\alpha)}) = - \log r(\alpha).
	\]
Since $r(\alpha) \to 0$ as $|\alpha| \to +\infty$, we conclude that $r_\varphi$ cannot be uniformly bounded as one varies the function $\varphi$. 
\end{example}

\begin{example}
\label{Ex: Pathology}
	Finally, we give an example to illustrate that the inequality in Proposition~\ref{Prop: Ramification Bounded by Tau} can be strict. Suppose $k$ has characteristic zero and residue characteristic $p > 0$. Consider the rational function
		\[
			\varphi(z) = \frac{z^{p+1} + p}{z}.
		\]
Then
	\[
		A_\varphi(z,w) =  \sum A_\ell(w) z^\ell =   p(w^{p+1} - 1) + \sum_{1 \leq \ell \leq p} \binom{p+1}{\ell+1}w^{p+1 - \ell} z^\ell.
	\]
The critical points of $\varphi$ are $\Crit(\varphi) = \{\infty, 1, \zeta_{p+1}, \ldots, \zeta_{p+1}^p\}$, where $\zeta_{p+1}$ is a primitive $(p+1)^{\mathrm{st}}$ root of unity. 
Then the disk $\DD(0,1)^-$ does not intersect the connected hull of the critical points. 


For $x = \zeta_{0,1}$, we see that $\vec{0} \in E_{\varphi,x}' \smallsetminus E_{\varphi, x}$ since $0$ is a root of $A_\ell(w)$ for $\ell = 1, \ldots, p$. The reduction of $\varphi$ is $\widetilde{\varphi}(z) = z^p$, so that $\zeta_{0,1} \in \Ram_\varphi$. Hence
	\begin{equation*}
			\tau_\varphi(\zeta_{0,1})  
				= \max_{\ell \geq 1} \log_{q_k} \left| \frac{1}{p}\binom{p+1}{\ell + 1}\right|^{1 / \ell} 
				= \frac{1}{p-1}.
	\end{equation*}
	
We claim that 
	\begin{equation}
	\label{Eq: tau comparison}
		\max_{\zeta \in \DD(0,1)^- \cap \Ram_\varphi} \rho(\zeta, \zeta_{0,1}) = \frac{1}{p+1} < \tau_\varphi(\zeta_{0,1}).
	\end{equation}
 To see it, we use the function $\mathfrak{t}_\varphi$ defined in \S\ref{Sec: Alternate}. Any $\zeta \in \HH \cap \DD(0,1)^-$ of type~II may be written as $\zeta = \zeta_{a,r}$ with $r = |a|$ for some $a \in k$. One calculates that
	\begin{align*}
			\mathfrak{t}_\varphi(\zeta_{a,r}) &= \log_{q_k} r + \log_{q_k} \max_{1 \leq \ell \leq p}
				\left| \frac{A_\ell}{A_0}(\zeta_{a,r}) \right|^{1 / \ell} \\
				&= \log_{q_k} r + \log_{q_k} \max\left\{ |p|^{-1/(p-1)} \cdot r^{2 / (p-1)}, 
						 |p|^{-1/p} \cdot r^{1 / p} \right\} \\
				&=\begin{cases}
							 \log_{q_k}  |p|^{-1/(p-1)} \cdot r^{(p+1) / (p-1)} & \text{if } r \geq |p|^{1 / (p+1)} \\
							 \log_{q_k}  |p|^{-1/p} \cdot r^{(p+1)/p} & \text{if } r < |p|^{1 / (p+1)}
					\end{cases} \\
				& \begin{cases}
						> 0 & \text{if } r > |p|^{1 / (p+1)} \\
						= 0 & \text{if } r = |p|^{1 / (p+1)} \\
						< 0 & \text{if } r < |p|^{1 / (p+1)}. 
					\end{cases}
	\end{align*}
By Propositions~\ref{Prop: Tau properties}\eqref{Item: Vanishing} and~\ref{Prop: Alternate}, we see that 
	\[
		\DD(0,1)^- \cap \Ram_\varphi = \{\zeta \in \DD(0,1)^-:  |p|^{1 / (p+1)} \leq \diam(\zeta)  < 1\},
	\]
which implies~\eqref{Eq: tau comparison} holds.

\end{example} 

\medskip
	
\noindent \textbf{Acknowledgments}
This work was made possible by a National Science Foundation Postdoctoral Research Fellowship. Many thanks go to Bob Rumely for his enthusiasm during the discovery of these results, and for his insightful comments on an earlier draft of this manuscript. The anonymous referee also deserves acknowledgement for several suggested improvements to the exposition. 


\bibliographystyle{plain}
\bibliography{xander_bib}

\end{document}